\documentclass[11pt]{amsart}

\usepackage{amssymb}
\usepackage{amsmath}
\usepackage{amsfonts}
\usepackage{graphicx}
\usepackage{amsthm}
\usepackage{enumerate}
\usepackage[mathscr]{eucal}
\usepackage{mathrsfs}
\usepackage{verbatim}
\usepackage{xcolor}
\usepackage{hyperref}
\usepackage{wrapfig}
\usepackage{color}

\makeatletter
\@namedef{subjclassname@2010}{%
  \textup{2010} Mathematics Subject Classification}
\makeatother

\numberwithin{equation}{section}

\theoremstyle{plain}
\newtheorem{theorem}{Theorem}[section]
\newtheorem{lemma}[theorem]{Lemma}
\newtheorem{proposition}[theorem]{Proposition}

\newtheorem{defn}[theorem]{Definition}

\theoremstyle{plain}

\numberwithin{equation}{section}

\theoremstyle{remark}

\DeclareMathOperator{\area}{Area}

\DeclareMathOperator{\rank}{rank}

\begin{document}
\date{\today}

\title[Surface with prescribed increasing spectrum and area]
{Prescription of finite Dirichlet eigenvalues and area on surface with boundary}

\author{Xiang He}
\thanks{Partially supported by  National Key R and D Program of China 2020YFA0713100, and
by NSFC no. 12171446, 11721101.}
\address{School of Mathematical Sciences\\
University of Science and Technology of China\\
Hefei, 230026\\ P.R. China\\}
\email{hx1224@mail.ustc.edu.cn}

\begin{abstract}
  In the present paper, we consider Dirichlet Laplacian on compact surface. We show that for a fixed surface with boundary $X$, a finite increasing sequence of real numbers $0<a_1<a_2<\cdots<a_N$ and a positive number $A$, there exists a metric $g$ on $X$ such that for any integer $1\leq k\leq N$, we have $\lambda_k^\mathcal{D}(X,g)=a_k$ and $\area(X,g)=A$.
\end{abstract}


\maketitle

\section{Introduction}

The problem of constructing Riemannian metrics on a given smooth manifold with prescribed (finitely many) Laplacian eigenvalues was first studied by Y. Colin de Verdi\`ere in \cite{CV0}, in which he showed that for any closed smooth manifold $M$ of dimension at least three, and any finite  sequence of positive numbers $0=a_1<a_2 \le a_3 \le \cdots \le a_N$, there is a Riemannian metric $g$ on $M$ with $\lambda_i(M, g)= a_i (1 \le i \le N)$.  Similar results of prescribing eigenvalues in various different settings have been proved by many authors, see for example \cite{MD}, \cite{PJ}, \cite{PJ12}, \cite{PJ14} and \cite{HW}.

The rough idea behind the constructions of  Riemannian metrics with prescribed eigenvalues is as  follows: one first construct a discrete graph with prescribed eigenvalues (or construct such a hyperbolic surface modeled on the graph) and embed it into the target manifold, then one  ``thicken" the graph or surface and apply a stability argument to conclude the existence of the demanded Riemannian metric.  In the case of an arbitrary finite sequence of $N$ positive numbers, the complete graphs of order $N$ are used to ensure the stability argument. As a result, one can see that the same argument fails for a surface: In general one cannot embed a complete graph of order $N$ into a given surface unless the genus of the surface is large enough.

In fact, there is such a difference between eigenvalues of Riemannian manifolds of dimension at least three and  of surfaces. According to Y. Colin de Verdi\`ere's result alluded to above, for any closed manifold of dimension at least three, each eigenvalue could have arbitrarily large multiplicity. On the other hand, for closed surfaces it was first discovered by S Y. Cheng (\cite{Cheng}) that each eigenvalue has a  multiplicity upper bound (depending on the topology of the surface). His result has been improved by many authors, see for example \cite{Bes}, \cite{CV0}, \cite{Nad} and \cite{Sev} etc.

On the other hand, if the given sequence of positive numbers is strictly increasing, namely $0=a_1<a_2 < a_3 < \cdots < a_N$, then instead of using complete graphs, one can use star graphs to carry out the stability argument. As a result Y. Colin de Verdi\`ere proved in \cite{CV0} that any strictly increasing finite sequence of positive numbers can be realized as the first $N$ eigenvalues of any closed surface.

In this short paper we study the problem of constructing Riemannian metrics on a smooth surface with boundary with prescribed Laplacian eigenvalues. There are two natural boundary conditions that have been studied extensively for Laplacian defined on manifolds with boundary, namely the Dirichlet condition and the Neumann condition. It turns out that the Neumann case is very similar to the boundary-less case, and thus we will only study the case of prescribing eigenvalues of the Dirichlet Laplacian.

The main result in this paper is

\begin{theorem}\label{mthm}
Let $X$ be a smooth surface with nonempty boundary, $0<a_1<a_2<\cdots<a_N$ be a given sequence of real numbers and $A$ be a given positive number. Then there exists a Riemannian metric $g$ on $X$ such that for any integer $1\leq k\leq N$, we have $\lambda_k^\mathcal{D}(X,g)=a_k$ and $\area(X,g)=A$.
\end{theorem}

The analogous problem  for  manifolds (with boundary) of dimension at least three has been studied recently in \cite{HW}, in which we showed that not only one can prescribe the first $N$ Dirichlet eigenvalue to be an arbitrary sequence $0<a_1<a_2 \le \cdots \le a_N$, but also one can prescribe the volume. We make two comments for the current setting:
\begin{itemize}
\item As in the boundary-less case, one can show that there is a multiplicity upper bound for each  Dirichlet eigenvalue of surfaces which depends on the topology of the surface (both the Euler characteristic $\chi(X)$ and the number of boundary components $b$).  For example, according to \cite{Ber}, for any surface with $\chi(X)+b<0$,  the $k$th Dirichlet eigenvalue has multiplicity no more than $2k-2(\chi(M)+b)+1$. So  we only construct Riemannian metrics with simple first $N$ eigenvalues. It is a natural question to ask whether it is possible to construct a Riemannian metric that attains (or gets close to) Berdnikov's bound.

\item The metric we will construct can also have arbitrarily prescribed area. Although in the case of manifold of dimension at least three, the same result holds for both closed eigenvalues problem (which was first proved by J. Lohkamp in \cite{L} via a doubling surface technique) and the Dirichlet eigenvalue problem (which was proven in \cite{HW}), the situation is different for surfaces.  In fact, for closed surfaces one cannot prescribe the first eigenvalue and the area at the same time. For example, it is recently proved by M. Karpukhin, N. Nadirashvili, A. Penskoi and I. Polterovich in \cite{KNPP} that  for any Riemannian metric with area one on the sphere $S^2$, the $k$th nonzero eigenvalue is no more than $8k\pi$, extending earlier results of J. Hersch \cite{Her} (for $k=1$), N. Nadirashvili \cite{Nad1} (for  $k=2$) and  N. Nadirashvili, Y.Sire \cite{NS} (for $k=3$). Similar upper bounds also exist for other closed surfaces:  P. Yang and S. T. Yau first proved in  \cite{YY} that for any oriented closed surface of genus $\gamma$ with area one, the first nonzero eigenvalue is no more than $8\pi(\gamma+1)$ and their result was extended  to non-orientable surfaces by  M. Karpukhin in \cite{Kar}  and to  higher eigenvalues by A. Hassannezhad in  \cite{Hass}.
So our result shows that there is another different nature between the Dirichlet eigenvalues for surfaces with boundary and  eigenvalues of closed surfaces, since for Dirichlet eigenvalues we can simultaneously prescribe the first $N$ eigenvalues and  the area.
\end{itemize}

\par 

We  briefly describe the idea of the proof, which is mainly inspired by Y. Colin de Verdi\`ere's work  in \cite{CV0}: 
We first  will use the ``adding handles" technique in Section 3 of \cite{CV0} to reduce the problem of prescribing finite Dirichlet spectrum on an arbitrary surface (with boundary) $X$ to the problem of prescribing finite Dirichlet spectrum on the disk, and we solve the problem on the disk by constructing star graphs  (with $N$ interior points and one  boundary point) with prescribed eigenvalues of the combinatorial Laplacian   and apply a stability argument.
The metric constructed in this way could have  arbitrarily small area. To  increase the area to  the prescribed value, we apply a new idea (which was also used in our earlier work \cite{HW}) of ``attaching rectangles"  to the boundary of the surface: by carefully choosing  the length of the sides of rectangle $R$ and the length of the intersection of $\partial R$ and $\partial X$,
we can show that the first $N$ eigenvalues of the Dirichlet Laplacian on $(X,h) \cup R$ and the first $N$ eigenvalues of the Dirichlet Laplacian on $(X,h)$ could be very close, and thus Theorem \ref{mthm} follows from another stability argument.

\textbf{Acknowledgments.} The author would like to thank his advisor, Zuoqin Wang, for numerous help during various stages of the work.

\section{The proof of Theorem \ref{mthm}}

The proof of the first part of Theorem \ref{mthm}, namely,  there exists a Riemannian metric $h$ on $X$ such that for any integer $1\leq k\leq N$, $\lambda_k^\mathcal{D}(X,g)=a_k$, is divided into three steps:
\begin{itemize}
  \item[Step 1.] reduce the problem to the case that the surface $X$ is a disk,
  \item[Step 2.] convert the problem of prescribing the first $N$ eigenvalues of the Dirichlet Laplacian on the disk to the problem of prescribing the eigenvalues of the combinatorial Laplacian on the star graph $E_{N,1}$ with $N$ interior points and one  boundary point,
  \item[Step 3.] prescribe the eigenvalues of the combinatorial Laplacian on the star graph  $E_{N,1}$.
\end{itemize}
Those three steps will be fulfilled in subsections 2.1-2.3 and for convenience, we will start with the third step. Then, we will prove the second part of Theorem \ref{mthm} in subsection 2.4.
\subsection{Eigenvalues on the star graph with one point boundary}

\noindent

For any  positive integer  $N$, we consider the star graph with $N$ interior points and one  boundary point, $E_{N,1}=(V,E)$, where $V=\{u_0\}\cup\{v_0,v_1,\cdots,v_{N-1}\}$ is the vertex set of $E_{N,1}$ and $E=\{(v_0,u_0)\}\cup\{(v_0,v_i),1\leq i\leq N-1\}$ is the edge set of $E_{N,1}$.

\par Let
 \begin{equation}
 H=\{f:V\to\mathbb{R}|\ f(u_0)=0\}.
 \end{equation}
For a given measure $\mu=\sum_{i=0}^{N-1}\mu_i\delta(v_i)$ $(\mu_i>0)$ on $E_{N,1}$, define an inner product $\langle\ ,\ \rangle_\mu$ on $H$ via
 \begin{equation}
 \langle f,g\rangle_\mu=\sum_{i=0}^{N-1}\mu_if(v_i)g(v_i), \ \forall f,g\in H.
 \end{equation}
For a given edge weight function
\[\Theta:E\to\mathbb{R}_{>0},\ (v_0,v_i)\mapsto\theta_i,\ (v_0,u_0)\mapsto\theta\]
(which will be viewed as a vector in $\mathbb{R}_{>0}^{N}$ in the following discussion), define a quadratic form $q_\Theta$ on $H$ via
 \begin{equation}
 q_\Theta(f)=\theta f(v_0)^2+\sum_{i=1}^{N-1}\theta_i(f(v_0)-f(v_i))^2.
 \end{equation}
Let $\Delta_{\mu,\Theta}$ be the combinatorial Laplacian associated with $(H,\mu,q_\Theta)$, namely
 \begin{equation}\label{1}
 \begin{aligned}
 \Delta_{\mu,\Theta}(f)(v_i)&=\frac{\theta_i}{\mu_i} (f(v_i)-f(v_0)),\ 1\leq i\leq N-1\\
 \Delta_{\mu,\Theta}(f)(v_0)&=\frac{\theta}{\mu_0}f(v_0) +\frac{1}{\mu_0}\sum_{i=1}^{N-1}\theta_i(f(v_0)-f(v_i)).
 \end{aligned}
 \end{equation}
For more background of the combinatorial Laplacian on graphs with boundary, one can see Chapter 8 of \cite{FRKC}.

Consider the eigenvalues of $\Delta_{\mu,\Theta}$, which will be denoted as
 \begin{equation}\label{eigfseq}
 0<\lambda_1(\mu,\Theta)<\lambda_2(\mu,\Theta)\leq\cdots\leq\lambda_N(\mu,\Theta).
 \end{equation}
By definition $\Delta_{\mu,\Theta}=\Delta_{t\mu,t\Theta}$, so we may normalize $\mu$ and $\Theta$ simultaneously so that  $\mu_0=1$, in which case the eigenvalues $\lambda = \lambda_k(\mu,\Theta)$  are exactly the solutions to the equation
 \begin{equation}
 1=\frac{\theta}{\lambda}+\sum_{i=1}^{N-1}\frac {\theta_i}{\lambda-\frac{\theta_i}{\mu_i}}.
 \end{equation}
It turns out that one can prescribe the eigenvalues (\ref{eigfseq}) to be any strictly increasing sequence by choosing suitable $\mu$ and $\Theta$:
\begin{proposition}\label{graph}
For any sequence $0<a_1<a_2<\cdots<a_N$, there exists a measure $\mu$ on $E_{N,1}$ and an edge weight function $\Theta:E\to\mathbb{R}_{>0}$ such that the eigenvalues of $\Delta_{\mu,\Theta}$ are precisely $a_i$'s.
\end{proposition}

\begin{proof} 
 Take $\mu_0=1$ and choose any $b_1, \cdots, b_{N-1}$ satisfying
\[0<a_1<b_1<a_2<b_2<\cdots<a_{N-1} <b_{N-1}<a_N.\]
Let $-\theta$ be the residue of the rational function
 \begin{equation*}
 R(\lambda)=\frac{\prod_{i=1}^N(\lambda-a_i)}{\lambda\prod_{i=1}^{N-1}(\lambda-b_i)}
 \end{equation*}
at $\lambda=0$ and $-\theta_i$ be the residue of $R(\lambda)$ at $\lambda=b_i$ for all $1\leq i\leq N-1$. Finally take $\mu_i=\frac{\theta_i}{b_i}$ for all $1\leq i \leq N-1$. It is easy to check that the $\mu$ and $\Theta$ constructed by this way are what we need.
\end{proof}

For similar result on the star graph, one can see \S4 of \cite{CV1}. Next, we consider the regularity of $\lambda_i(\mu,\Theta)$ on $\{\theta,\theta_i,\mu_j|\ 1\leq i\leq N-1,\ 0\leq j\leq N-1\}$ in the region that each eigenvalue is simple:

\begin{proposition}\label{ni}
For any $0<\lambda_1<\lambda_2<\cdots<\lambda_N$, the map
 \begin{equation}\label{Phi}
 \Phi:\mathbb{R}_{>0}^{2N}\to\mathbb{R}_{> 0}^N,\ \{\theta,\theta_i,\mu_j|\ 1\leq i\leq N-1,\ 0\leq j\leq N-1\}\mapsto\{\lambda_j\}_{1\leq j\leq N}
 \end{equation}
is a submersion on a neighborhood of $\Phi^{-1}(\lambda_1,\cdots,\lambda_N)$.
\end{proposition}

\begin{proof}
We first prove the smoothness of  $\Phi$  on a neighborhood of $\Phi^{-1}(\lambda_1,\cdots,\lambda_N)$. By definition, $\lambda_i$'s are solutions of a degree $N$ polynomial,
\[P_c(\lambda)=\lambda^N+\sum\nolimits_{i=0}^{N-1}c_i\lambda^i,\]
where $c=(c_0, \cdots c_{N-1})$ depends smoothly on $\mu, \Theta$. Note that since $P_c$ has $n$ distinct real roots, for any $c'$ in a small neighborhood of $c$, the polynomial $P_{c'}$ also has $n$ distinct real roots. So it is enough to prove that the map
\begin{equation}\label{loc-diff}
\{c_i\}_{0\leq i\leq N-1}\mapsto\{\lambda_j\}_{1\leq j\leq N}
\end{equation}
is a local diffeomorphism in a neighborhood of $c$ under the condition that $P_c$ has roots  $0<\lambda_1<\cdots<\lambda_N$.

We use  induction on $N$. For $N=1$, the conclusion is obvious since the map is $\lambda_1=-c_0$. Assume that for $N=d-1$, the map (\ref{loc-diff}) is a local diffeomorphism when $0<\lambda_1<\cdots<\lambda_{d-1}$. For $N=d$, write $P_c(\lambda)$ as
 \begin{equation*}
 P_c(\lambda)=(\lambda-\lambda_1)(\lambda^{d-1}+\sum_{i=0}^{d-2}b_i\lambda^i).
 \end{equation*}
If we write $b_{d-1}=1$ and $b_{-1}=0$, then for any $0 \le i \le d-1$,
\[c_i=b_{i-1}-\lambda_1 b_i.\]
Under the assumption that $\lambda_1$ is a simple root of $P_c$, the map
 \begin{equation*}
 F:\mathbb{R}_{>0}\times\mathbb{R}^{d-1}\to\mathbb{R}^d,\ \{\lambda_1,b_i\}_{0\leq i\leq d-2}\mapsto\{c_i\}_{0\leq i\leq d-1}
 \end{equation*}
is a local diffeomorphism since
\[\det(\mathrm{d}F)=-(\lambda_1^{d-1}+\sum_{i=0}^{d-2}b_i\lambda_1^i)\neq 0.\]
Denote the local inverse of $F$ to be
 \begin{equation*}
 G:U_a \subset\mathbb{R}^d\to\mathbb{R}_{>0}\times\mathbb{R}^{d-1},\ \{c_i\}_{0\leq i\leq d-1}\mapsto \{\lambda_1,b_i\}_{0\leq i\leq d-2}.
 \end{equation*}
 Since  by reduction, the map
 \begin{equation*}
 H:\mathbb{R}^{d-1}\to\mathbb{R}_{>0}^{d-1},\ \{b_i\}_{0\leq i\leq d-2}\mapsto\{\lambda_j\}_{2\leq j\leq d}
 \end{equation*}
is a local diffeomorphism, the map $\{c_i\}_{0\leq i\leq d-1}\mapsto\{\lambda_j\}_{1\leq j\leq d}$ is a local diffeomorphism when $0<\lambda_1<\cdots<\lambda_d$.

To prove that $\Phi$ is a submersion, it is enough to notice that when  restricted to the subspace
\[L_b=\{ \theta,\theta_i,\mu_j\ |\ \mu_0=1,\ \frac{\theta_i}{\mu_i}=b_i, 1 \le i \le N-1\} \subset \{\theta,\theta_i,\mu_j|\ 1\leq i\leq N-1,\ 0\leq j\leq N-1\},\]
the map $\Phi|_{L_b}$ is invertible  on a neighborhood of $\Phi^{-1}(\lambda_1,\cdots,\lambda_N)$, with inverse the map constructed in the proof of Proposition \ref{graph}, namely
\[
(a_1, a_2, \cdots, a_N) \mapsto (-\mathrm{Res}_{\lambda=0}R(\lambda), -\mathrm{Res}_{\lambda=b_{i}}R(\lambda), -\frac{\mathrm{Res}_{\lambda=b_{i}}R(\lambda)}{b_i}).
\]
This completes the proof.
\end{proof}

\subsection{Reduce the problem to the case of disk}\label{simplify}

\par\noindent

\par The purpose of this subsection is to reduce the problem to the case that the surface $X$ is a disk. First we introduce the conception of a stable map:
\begin{defn}Let $B$ be a closed ball in $\mathbb{R}^M$ and $f:B\to\mathbb{R}^N$ be a continuous map. We say $f$ is \textbf{stable} at point $y\in f(B)$ if there exists an $\alpha>0$ such that for any map $g:B\to\mathbb{R}^N$ with $\|g-f\|_{C(B)}<\alpha$, there is a point $x\in\mathrm{int} (B)$ with $g(x)=y$.
\end{defn}

\par For any surface with boundary  $X$, let $\mathcal{M}(X)$ be the space of smooth Riemannian metrics on $X$. Given any sequence $0<a_1<a_2<\cdots<a_N$, suppose that there exists a $C^1$ map
\[f:B\to\mathcal{M}(X)\]
(where $B$ is a closed ball centered at 0) such that the associated map 
 \begin{equation*}
 \tilde f:B\to\mathbb{R}^N,\ \tilde f(p)=\big(\lambda_1^\mathcal{D}(X,f(p)),\cdots,\lambda_N^\mathcal{D}(X,f(p))\big)
 \end{equation*}
satisfies
\begin{itemize}
\item $\tilde f(0)=(a_1, \cdots, a_N)$,
\item $\tilde f$ is stable at $(a_1,\cdots,a_N)$.
\item $\tilde f$ is Lipschitz and there exists a constant $L>0$ such that 
\begin{equation}\label{L}
|\tilde f(p)-\tilde f(q)|\leq L|p-q|,\qquad \forall p,q\in B.
\end{equation}
\end{itemize}
The last item is ensured by the $C^1$ assumption of $f$.

We will show
\begin{proposition}\label{p3.2}
Under the previous assumption, for each of the surfaces
\begin{itemize}
\item $X'=X \# D$ (i.e. $X$ with one puncture),
\item $\widetilde X=X \# \mathbb T^2$ (i.e. $X$ with one handle),
\item $\widehat X=X \# \mathbb{RP}^2$ (i.e. $X$ with one cross hat),
\end{itemize}
there exists a family of metrics such that the associated map of eigenvalues is still stable at $(a_1,\cdots,a_N)$.
\end{proposition}
\begin{proof}
First we realize $X'$ as $X_\varepsilon=X\setminus B_\varepsilon(x)$, where  $x\in X$ is a point and $B_\varepsilon(x)$ is the $\varepsilon$-geodesic disk centered at $x$  with respect to the corresponding metric. 
According to Theorem 1.4 in \cite{GC}, one has
 \begin{equation}\label{coneigen}
 \lim_{\varepsilon\to 0}\lambda_k^\mathcal{D}(X_\varepsilon,f(p))=\lambda_k^\mathcal{D}(X,f(p)),\ \forall k\geq 1
 \end{equation}
for any $p\in B$. By the $C^1$ assumption of $f$, the map
\[
\tilde f_\varepsilon:B\to\mathbb{R}^N,\ \tilde f_\varepsilon(p)=\big(\lambda_1^\mathcal{D}(X_\varepsilon,f(p)),\cdots,\lambda_N^\mathcal{D}(X _\varepsilon,f(p))\big)
\]
is Lipschitz and one can assume that 
\begin{equation}\label{epsL}
|\tilde f_\varepsilon(p)-\tilde f_\varepsilon(q)|\leq L|p-q|,\qquad \forall p,q\in B
\end{equation}
where $L$ is the same constant in (\ref{L}). Then by (\ref{coneigen}) and (\ref{epsL}), $\tilde f_\varepsilon$ is uniformly converges to $\tilde f$ as $\varepsilon$ goes to zero and further, by the stability of $\tilde f$ at $(a_1,\cdots,a_N)$, when $\varepsilon$ is small enough, there exists $q\in \mathrm {int} (B)$ such that
 \begin{equation*}
 \lambda_k^\mathcal{D}(X_\varepsilon,f(q))=a_k,\ \forall 1\leq k\leq N
 \end{equation*}
and the map $\tilde f_\varepsilon$ is also stable at $(a_1,\cdots,a_N)$.

Next we realize $\widetilde X$ as the surface obtained by gluing $\partial B_\varepsilon(x_1)$ and $\partial B_\varepsilon(x_2)$ of $X\setminus(B_\varepsilon(x_1)\cup B_\varepsilon(x_2))$,
where $B_\varepsilon(x_1)$ and $B_\varepsilon(x_2)$ are non-intersecting $\varepsilon$-geodesic disks centered at $x_1$ and $x_2$ respectively. Since the map $\tilde f$ is stable at $(a_1,\cdots,a_N)$, when slightly changing the metric $f(p)$ near $x_1$ and $x_2$, the map of eigenvalues associated to the new metric is still stable at  $(a_1,\cdots,a_N)$. So we may assume that the metric $f(p)$ is of the form
\[\mathrm{d}r^2+\varepsilon^2\mathrm{d}\theta^2,\quad  (r,\theta)\in(\varepsilon,\varepsilon+\rho)\times[0,2\pi)\]
on the annuli $B_{\varepsilon+\rho}(x_1)\setminus B_\varepsilon(x_1)$ and $B_{\varepsilon+\rho}(x_2)\setminus B_\varepsilon(x_2)$, where  $\rho$ is a small constant depending on $\varepsilon$.
After gluing $\partial B_\varepsilon(x_1)$ and $\partial B_\varepsilon(x_2)$ of $X\setminus(B_\varepsilon(x_1)\cup B_\varepsilon(x_2))$ for each $f(p)$, we get a family of metrics
\[k :B\to\mathcal{M}(\widetilde X).\]
Let $\{\lambda^\mathcal{N}_k(p,\varepsilon)\}_{k=1}^\infty$ be the eigenvalues of the Laplacian on
\[(X\setminus(B_\varepsilon(x_1)\cup B_\varepsilon(x_2)), f(p))\]
with the following boundary conditions:
 \begin{center}
 Dirichlet on $\partial X$ and Neumann on $\partial B_\varepsilon(x_1)\cup \partial B_\varepsilon(x_2)$
 \end{center}
and $\{\lambda^\mathcal{D}_k(p,\varepsilon)\}_{k=1}^\infty$ be the eigenvalues of the Dirichlet Laplacian on
\[(X\setminus(B_\varepsilon(x_1)\cup B_\varepsilon(x_2)), f(p)).\]
In the appendix, we will prove   $\lim_{\varepsilon \to 0}\lambda_k^\mathcal{N}(p,\varepsilon)=\lambda^\mathcal{D}_k(X,f(p))$. Again by  Theorem 1.4 in \cite{GC},  $\lim_{\varepsilon \to 0}\lambda_k^\mathcal{D}(p,\varepsilon)=\lambda^\mathcal{D}_k(X,f(p))$. On the other hand, by the standard min-max characterization of eigenvalues,
 \begin{equation*}
 \lambda_k^\mathcal{N}(p,\varepsilon)\leq\lambda_k^\mathcal{D}(\widetilde X, k(p))\leq\lambda_k^\mathcal{D}(p,\varepsilon).
 \end{equation*}
So similar to the proof of the first part, the map
\[\tilde{k}:B\to\mathbb{R}^N, \quad  \tilde{k}(p)=\big(\lambda_1^\mathcal{D}(\widetilde X,k(p)),\cdots,\lambda_N^\mathcal{D}( \widetilde X,k(p))\big)\]
is also stable at $(a_1,\cdots,a_N)$ when $\varepsilon$ is small enough.

Finally we realize $\widehat X$ as the surface obtained from $X \setminus B_\varepsilon(x_1)$ by gluing each pair of opposite points on the circle $\partial B_\varepsilon(x_1)$
, where $B_\varepsilon(x_1)$ is the $\varepsilon$-geodesic disk centered at $x_1$. Similar to the case of adding handle, we can assume that the metric $f(p)$ is of the form
\[\mathrm{d}r^2+\varepsilon^2\mathrm{d}\theta^2,\quad  (r,\theta)\in(\varepsilon,\varepsilon+\rho)\times[0,2\pi)\]
on the annuli $B_{\varepsilon+\rho}(x_1)\setminus B_\varepsilon(x_1)$, where  $\rho$ is a small constant depending on $\varepsilon$. After gluing opposite points on $\partial B_\varepsilon(x_1)$, we get a family of metrics
\[l :B\to\mathcal{M}(\widehat  X).\]
Let $\{\Lambda^\mathcal{N}_k(p,\varepsilon)\}_{k=1}^\infty$ be the eigenvalues of the Laplacian on
\[(X\setminus B_\varepsilon(x_1), f(p))\]
with the following boundary conditions:
 \begin{center}
 Dirichlet on $\partial X$ and Neumann on $\partial B_\varepsilon(x_1)$
 \end{center}
and $\{\Lambda^\mathcal{D}_k(p,\varepsilon)\}_{k=1}^\infty$ be the eigenvalues of the Dirichlet Laplacian on
\[(X\setminus B_\varepsilon(x_1), f(p)).\]
Similarly, one can prove that
\[\lim_{\varepsilon\to 0}\Lambda^\mathcal N_k(p,\varepsilon)=\lim_{\varepsilon\to 0}\Lambda^\mathcal D_k(p,\varepsilon)=\lambda_k^\mathcal D(X,f(p))\]
and
\[\Lambda^\mathcal N_k(p,\varepsilon) \leq \lambda_k^\mathcal D(\widehat X,l(p)) \leq \Lambda^\mathcal D_k(p,\varepsilon).\]
So similar to the proof of the first part, the map
\[\tilde{l}:B\to\mathbb{R}^N, \quad  \tilde{l}(p)=\big(\lambda_1^\mathcal{D}(\widehat X,l(p)),\cdots,\lambda_N^\mathcal{D}( \widehat X,l(p))\big)\]
is also stable at $(a_1,\cdots,a_N)$ when $\varepsilon$ is small enough.
\end{proof}

Since every  surface with boundary is topologically of the form
\[D \# \mathbb{T}^2 \# \cdots \#\mathbb{T}^2\#D\# \cdots \# D \quad \text{or}\quad D \# \mathbb{RP}^2 \# \cdots \# \mathbb{RP}^2 \# D\# \cdots \# D,\] it remains to prove the first part of Theorem \ref{mthm} for  the disk $D$.

\subsection{Proof of the first part of Theorem \ref{mthm} for  the disk} \label{con mt}

\par\noindent

\par By Proposition \ref{graph}, for any $0<a_1<\cdots<a_N$, there exists $\mu$ and $\Theta$ on the graph $E_{N,1}$ such that $\lambda_i(\mu,\Theta)=a_i$ for all $1\leq i\leq N$. Then by scaling a constant on both $\mu$ and $\Theta$, we can assume that $\mu_0\geq N+1$ and $\mu_i\geq 2$ for all $1\leq i\leq N-1$. Now we construct a metric on the disk $D$ using $\mu$ and $\Theta$:
\begin{itemize}
  \item $Z_i^\varepsilon=[0,\mathrm{arccosh}\frac{1}{\theta_i\pi\varepsilon}] \times\mathbb{R}/\mathbb{Z}$ with metric $\mathrm{d}x^2+(\theta_i\pi\varepsilon\cosh x)^2\mathrm{d}\theta^2$,
  \item[$\bullet$]$\widehat{M_i^\varepsilon}$ = the surface $D$ endowed with a metric whose boundary length is 1 and $\area(\widehat{M_i^\varepsilon})+\area(Z^\varepsilon_i)=\mu_i$,
  \item[$\bullet$]$Z^\varepsilon=[0,\mathrm{arccosh\frac{2}{\theta\pi\varepsilon}}]\times\mathbb {R}/\mathbb{Z}$ with metric $\mathrm{d}x^2+(\frac{\theta\pi\varepsilon}{2}\cosh x)^2\mathrm{d}\theta^2$,
  \item[$\bullet$]$\widehat{M^\varepsilon_0}$ = the surface which homeomorphic to $D$ with $N-1$ small disks removed, endowed with a metric so that the boundary circles are all of length 1, and $\area(\widehat{M^\varepsilon_0})+\area(Z^\varepsilon)+\sum_{i=1}^{N-1}\area(Z^\varepsilon _i)=\mu_0$.
\end{itemize}
By carefully choosing metrics on $\widehat{M_i^\varepsilon}$ and $\widehat{M^\varepsilon_0}$, one may assume
\begin{itemize}
  \item the surface $M_i$ constructed by gluing $Z^\varepsilon_i$ to $\widehat{M_i^\varepsilon}$ along the boundary circle of length 1 is a smooth surface, see the graph \\
      \includegraphics[scale=0.4]{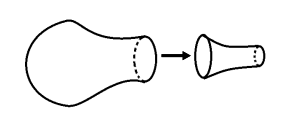}
  \item the surface $M_0$ constructed by gluing $Z^\varepsilon, Z^\varepsilon_1, \cdots, Z^\varepsilon_{N-1}$ to $\widehat{M^\varepsilon_0}$ along the boundary circles of length 1 is a smooth surface.
\end{itemize}
Note that each $M_i$ has a boundary circle of length $\theta_i\pi\varepsilon$, while $M_0$ has $N-1$ boundary circles of lengths $\theta_1\pi\varepsilon, \cdots,  \theta_{N-1}\pi\varepsilon$ respectively. By gluing the corresponding boundary circles of the same length, we get the demanded metric on $D$, which will be denoted by \begin{equation}\label{hvarep}
h_\varepsilon=h_\varepsilon(\theta,\theta_i,\mu_i).
\end{equation}
In the construction, we can make $h_\varepsilon$ be $C^1$ with respect to the parameters $\{\theta,\theta_i,\mu_i\}$.

Now consider the function space
 \begin{equation}
 E^\varepsilon_0=\left\{f\in H^1_0(D)\left| \!\!\!\!\! \!\!\!
  \begin{array}{ll} & \text{there exist  } x_0 \text{ and } x_i \text{ such that } f|_{\widehat{M^\varepsilon_i}}\equiv x_i,  f|_{\widehat{M_0^\varepsilon}}\equiv x_0,\\
   & \text{and $f$ is harmonic on $Z^\varepsilon$ and each $Z^\varepsilon_i$}\end{array}\right.\right\},\\
 \end{equation}

Similar to Proposition 4.3 and 4.4 in \cite{HW}, one has
\begin{proposition}\label{p3.3}
Suppose $f_\varepsilon\in E_0^\varepsilon$ and $f_\varepsilon|_{\widehat{M^\varepsilon_i}}\equiv x_i$, $f_\varepsilon|_{\widehat{M_0^\varepsilon}}\equiv x_0$ are independent of $\varepsilon$, then
 \begin{equation}\label{f32}
 \lim_{\varepsilon\to0}\frac{\int_D f_\varepsilon^2\mathrm{d}V_{h_\varepsilon}}{\sum_{i=0}^{N-1}\mu_i x_i^2}=1
 \end{equation}
and
 \begin{equation}\label{f33}
 \lim_{\varepsilon\to0}\frac{\int_D|\nabla f_\varepsilon|^2\mathrm{d}V_{h_\varepsilon}}{\varepsilon\cdot(\theta x_0^2+\sum_{i=1}^{N-1}\theta_i(x_0-x_i)^2)}=1.
 \end{equation}
Moreover when $\varepsilon$ is small enough, there exists a uniform constant $C>0$ such that for any $f\in E_0^\varepsilon$, $g\in H_0^1(D)$,
 \begin{equation}
 |\int_D\nabla f\cdot\nabla g\mathrm{d}V_{h_\varepsilon}|\leq C\varepsilon\|f\|_{L^2(D,h_\varepsilon)}\|g\|_{H^1_0(D,h_\varepsilon)}.
 \end{equation}
\end{proposition}

\par Next, we prove that the eigenvalue $\lambda_{N+1}^\mathcal{D}(D,h_\varepsilon)$ have a uniform lower bound when $\varepsilon$ is small enough. This is based on the following two lemmas.
Consider the cylinder $T=[a,b]\times \mathbb{R}/\mathbb{Z}$ with the metric $\mathrm{d}r^2+l^2\cosh^2r\mathrm{d}\theta^2$.
\begin{lemma}[Lemma 3.2 in \cite{JTBD}] \label{l3.2} For any continuously differentiable function  $k$  which vanishes on $\partial T$,
 \begin{equation*}
 \frac{1}{4}\int_T k^2\leq\int_T |\nabla k|^2.
 \end{equation*}
\end{lemma}

\begin{lemma}[Lemma 3.3 in \cite{JTBD}]\label{l3.3} Suppose $b-a>2$ and let $S\subset T$ be the set of points of $T$ at distances less than or equal to 1 from $\partial T$. Then there exists $\eta>0$ such that for any $k \in H^1(T)$ satisfying
\[0<\int_T k^2=c<\infty, \quad \int_S k^2<\eta c, \quad \int_S|\nabla k|^2<\eta c,\]
we have
 \begin{equation*}
 \int_T|\nabla k|^2>\frac{c}{8}.
 \end{equation*}
\end{lemma}
For simplicity we denote $\lambda^\varepsilon_{N+1}=\lambda_{N+1}^\mathcal{D}(D,h_\varepsilon)$.
\begin{proposition}\label{lbd}
For $\varepsilon$  small enough, there exists a uniform $\beta>0$ such that $\lambda^\varepsilon_{N+1}>\beta$.
\end{proposition}

\begin{proof}Denote by
\begin{itemize}
  \item $\gamma_i$ the geodesic circle on $\partial M_0$ of length $\theta_i\pi\varepsilon$,
  \item $\gamma$ the geodesic circle on $\partial M_0$ of length $\frac{\theta\pi\varepsilon}{2}$.
\end{itemize}
Let
\begin{itemize}
  \item $b_i$ be the first nonzero Neumann eigenvalue of the surface $M_i$,
  \item $b$ be the second eigenvalue of mixed boundary-value Laplacian on $M_0$ with Neumann on each $\gamma_i$ and Dirichlet on $\gamma$.
\end{itemize}
Then by min-max principle, one gets
 \begin{equation}
 \lambda^\varepsilon_{N+1}\geq \min(b,\min_{1\leq i\leq N-1}b_i).
 \end{equation}
It remains to give a lower bound for $b$ and each $b_i$.

Let $\varphi_i$ be an eigenfunction associated to $b_i$ with $\|\varphi_i\|_{L^2}=1$. Then by Moser iteration in standard elliptic theory and the fact $\Delta_1\mathrm{d}\varphi_i=\mathrm{d}\Delta\varphi_i=b_i\mathrm{d}\varphi_i$ (where $\Delta_1$ is the Hodge Laplacian on 1-forms), there exists a uniform constant $C$ for $\varepsilon$ small enough such that
 \begin{equation}\label{3.6}
 \max_{\widehat{M^\varepsilon_i}}|\nabla\varphi_i|\leq C(1+b_i)\|\nabla\varphi_i\|_{L^2}=C(1+b_i)\sqrt{b_i}.
 \end{equation}
Let
 \begin{equation*}
 \bar\varphi_i=\frac{1}{\area(\widehat{M^\varepsilon_i})}\int_{\widehat{M^\varepsilon_i}} \varphi_i\mathrm{d}V_{h_\varepsilon}.
 \end{equation*}
Fix  a constant $a>0$ independent of $\varepsilon$ with
\[a\area(\widehat{M^\varepsilon_i})<\frac{1}{2} \text{ and } 4a^2<\frac{\eta}{2},\]
where $\eta$ is the constant in Lemma \ref{l3.3}.
\begin{itemize}
  \item  For the case $\bar\varphi_i>a$: If $b_i$ is very small, then by (\ref{3.6}), $\varphi_i$ will not change sign on $\widehat{M^\varepsilon_i}$. By Courant's nodal domain theorem, there exists a component of the nodal domain of $\varphi_i$ that lies in $Z^\varepsilon_i$. So by Lemma \ref{l3.2}, $b_i\geq\frac{1}{4}$ which   contradicts to the smallness of $b_i$. Thus $b_i$ has a uniform lower bound in this case.
  \item For the case $\bar\varphi_i\leq a$: reflect the function $\varphi_i|_{Z^\varepsilon_i}$ to get a function on the cylinder $\widetilde{Z^\varepsilon_i} =[-\mathrm{arccosh}\frac{1}{\theta_i\pi\varepsilon},\mathrm{arccosh}\frac{1}{\theta_i\pi\varepsilon}] \times\mathbb{R}/\mathbb{Z}$. If $b_i$ is very small, then the extended function  on $\widetilde{Z^\varepsilon_i}$ verifies all conditions in Lemma \ref{l3.3} with the constant $c>\frac{1}{2}$. So we get  $\int_{Z^\varepsilon_i}|\nabla\varphi_i|^2\mathrm{d}V_{h_\varepsilon}>\frac{1}{32}$ which contradicts to the smallness of $b_i$.
\end{itemize}
So $b_i$ has a uniform lower bound for $\varepsilon$ small enough.

For the lower bound of $b$, the proof is similar: Let $\psi$ be a $L^2$-normalized eigenfunction associated to $b$. By using Moser iteration, one can prove
\[\sup_{\widehat{M_0^\varepsilon}}|\nabla\psi|\to0\]
as $b\to0$. Thus the same argument as above implies that $b$ has a uniform lower bound for $\varepsilon$ small enough. This completes the proof. 
\end{proof}

The last ingredient for the proof is
\begin{lemma}[Lemma 1.1 in \cite{CC}]\label{l3.7}
Let $(H,|\cdot|)$ be a Hilbert space and $q$ be a positive quadratic form (whose domain $\mathcal{D}$ is dense in $H$) with discrete spectrum $\{v_k\}_{k=1}^\infty$. Let $F$ be an $N$-dimensional subspace of $\mathcal{D}$ and $\{v_j^F\}_{j=1}^N$ be the eigenvalues of the restriction of $q$ to $F$. Assume that $v_{N+1}\geq C_1\geq v_N$.
Then there exist constants $C(C_1,N)>0$ and $\varepsilon_0(C_1,N)>0$ such that if
\begin{equation}\label{qphig}
|q(\phi,g)|\leq\varepsilon|\phi||g|_{+1}, \qquad \forall g\in\mathcal{D}, \phi\in F
\end{equation}
holds for some $0<\varepsilon<\varepsilon_0(C_1,N)$, where $|g|_{+1}=\sqrt{|g|^2+q(g)}$, then
 \begin{equation*}
 v^F_j-C(C_1,N)\varepsilon^2\leq v_j\leq v_j^F,\ \forall 1\leq j\leq N.
 \end{equation*}
\end{lemma}

Now we are ready to prove the first part of the main theorem. Consider the map $\Phi^\varepsilon: \mathbb R^{2N}_{>0} \to \mathbb R^N$ given by
 \begin{equation}\label{phi-eps}
 \{\theta,\theta_i,\mu_j|\ 1\leq i\leq N-1,\ 0\leq j\leq N-1\}\mapsto\{\lambda_i^\mathcal{D}(D,\varepsilon h_\varepsilon)\}_{i=1}^N,
 \end{equation}
where $h_\varepsilon =h_\varepsilon(\theta,\theta_i,\mu_i)$ is the metric in (\ref{hvarep}).
Apply Lemma \ref{l3.7} to the case $H=L^2(D,h_\varepsilon)$, $\mathcal{D}=H^1_0(D)$, $F=E_0^\varepsilon$ and $q(f)=\int_D|\nabla f|^2\mathrm{d}V_{h_\varepsilon}$. By Proposition \ref{lbd}, we may take $C_1=\beta$. The condition (\ref{qphig}) follows from the last part of Proposition \ref{p3.3}. Thus by Lemma \ref{l3.7} together with the convergences (\ref{f32}) and (\ref{f33}), one gets the following convergence
 \begin{equation*}
 \lim_{\varepsilon\to0}\lambda_i^\mathcal{D}(D,\varepsilon h_\varepsilon)=\lambda_i,\quad 1\leq i\leq N,
 \end{equation*}
where $(\lambda_1,\cdots,\lambda_N)=\Phi(\theta,\theta_i,\mu_j)$   was given in (\ref{Phi}).
Since $\Phi$ is a submersion when $0<\lambda_1<\lambda_2<\cdots<\lambda_N$, for $\varepsilon$ small enough  the map $\Phi^\varepsilon$ is stable at $(a_1,\cdots,a_N)$ when  restricted to a small ball centered at $\{\theta,\theta_i,\mu_j\} \in \Phi^{-1}(a_1,\cdots,a_N)$.
This completes the proof of the first part of Theorem \ref{mthm}.

\subsection{Proof of the second part of Theorem \ref{mthm}}

\noindent

Then, we prove the second part of Theorem \ref{mthm}.

\begin{proof}[The Proof of the second part of Theorem \ref{mthm}]

\noindent

Similar to the proof of the first part of Theorem \ref{mthm}, one only needs to consider the case of the disk $D$. Take $\varepsilon$ small enough such that
\[\area (D,\varepsilon h_\varepsilon)<\frac{A}{2}\]
where $\varepsilon h_\varepsilon$ is the metric constructed in Subsection \ref{con mt}. Then attach a rectangle $R=[0,a]\times [0,b]$ to $D$ along a small boundary interval $I$ of length $c$, as illustrated below

\includegraphics[scale=0.5]{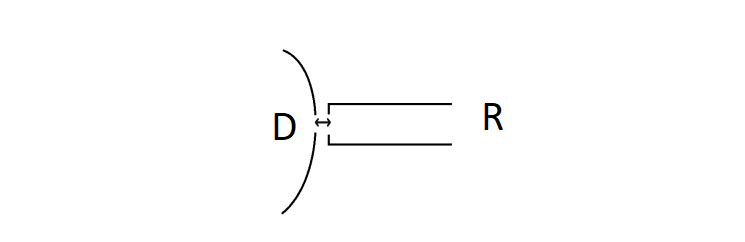}\\
where $a, b$ and $c$ will be carefully chosen below, with
\[ \area(X,\varepsilon h_\varepsilon)+ab=A.\]

Denote the resulting surface by $X_R=D\cup_I R$. Decompose
\[H^1_0(X_R)=\mathcal{H}_0\oplus\mathcal{H}_\infty,\]
where $\mathcal{H}_0=H^1_0(D)$ and
\[\mathcal{H}_\infty=\mathcal{H}_\infty(I)=\{\ f\in H^1_0(X_R)\ | \ \Delta f=0\text{ in }D\}.\]
By the proof of Lemma 5.2 in \cite{HW}, for any $M>0$, one has
\[ \int_D |\nabla f|^2 \mathrm{d} V_{\varepsilon h_\varepsilon} \ge M \int_D f^2 \mathrm{d}V_{\varepsilon h_\varepsilon},\ \forall f\in\mathcal{H}_\infty\]
for $c$ small enough.

Let $\lambda_1(c)$ be the first eigenvalue of the mixed boundary-valued Laplacian on $R$ with
\begin{center} Dirichlet on $\partial R\!\setminus\! I\;$ and $\;$ Neumann on $I$. \end{center}
By Lemma 5.2 in \cite{HW}, for  $c$  small enough one has
\[\lambda_1(c) \ge \frac{\lambda^\mathcal D_1(R)}{2}=\frac{\pi^2}{2}(\frac{1}{a^2}+\frac{1}{b^2}).\]
Moreover, one can make $\lambda^\mathcal{D}_1(R) \ge 2M$ by carefully choosing $a$ and $b$.

So one can make
\[ \int_{X_R} |\nabla f|^2 \mathrm{d} V =\int _D |\nabla f|^2 \mathrm{d} V_{\varepsilon h_\varepsilon} + \int_R |\nabla f|^2 \mathrm{d} V \ge M \int_{X_R} f^2 \mathrm{d} V,\ \forall f\in \mathcal{H}_\infty,\]
by carefully choosing $a$, $b$, $c$. By Theorem 2.3 in \cite{HW}, for any integer $k \ge 1$, one has
\begin{equation}\label{con-XR-X}
\lambda_k^\mathcal{D} (X_R) \to \lambda_k^\mathcal{D} (D,\varepsilon h_\varepsilon)
\end{equation}
as $M\to \infty$. One also can slightly change the metric on $X_R$ such that the new metric is smooth with the same area and the convergence (\ref{con-XR-X}) still holds. Then Theorem \ref{mthm} follows from the fact that $\Phi^\varepsilon$ (defined in (\ref{phi-eps})) is stable at $(a_1,\cdots,a_N)$ when  restricted to a small ball centered at $\{\theta,\theta_i,\mu_j\} \in \Phi^{-1}(a_1,\cdots,a_N)$.
\end{proof}

\section{Appendix}

\par The convergence of $\lambda_k^\mathcal{N}(p,\varepsilon)$ to $\lambda^\mathcal{D}_k(X,f(p))$ is just a corollary of the following theorem.

\begin{theorem}\label{ConN}Let $(M,g)$ be a compact Riemannian manifold with boundary and $\{\lambda_k^\mathcal{D}\}_{k=1}^\infty$ be the eigenvalues of the Dirichlet Laplacian on $(M,g)$. For $p\in \mathrm{int} M$, denote $M_\varepsilon=M\setminus B_\varepsilon(p)$ where $B_\varepsilon(p)$ is the $\varepsilon$-geodesic ball with center $p$. Let $\{\lambda_k(\varepsilon)\}_{k=1}^\infty$ be the eigenvalues of the mixed boundary-value Laplacian on $M_\varepsilon$ with
 \begin{center} Dirichlet condition on $\partial M$ and Neumann condition on $\partial B_\varepsilon(p)$. \end{center}
Then $\lim_{\varepsilon\to 0}\lambda_k(\varepsilon)=\lambda_k^\mathcal{D}$ for all $1\leq k<\infty$.
\end{theorem}

\par Before starting the proof, we make some notes. Let $\Delta_M$ be the Friedrich extension of the quadratic form $q(f)=\int_M|\nabla f|^2\mathrm{d}V_g$ with domain $H^1_0(M)$ and $\Delta_\varepsilon$ be the Friedrich extension of the quadratic form $q_\varepsilon(f)=\int_{M_\varepsilon}|\nabla f|^2\mathrm{d}V_g$ with domain
\[D_\varepsilon=\{f\in H^1(M_\varepsilon)|\ f|_{\partial M}=0\}.\]
For $u\in H^1_0(M)$, let $u_\varepsilon=u|_{M_\varepsilon}$ and $P_\varepsilon u_\varepsilon$ be the Harmonic extension  of $u$ with
 \begin{equation*}
 \begin{cases}
 P_\varepsilon u_\varepsilon=u&\text{ on }M_\varepsilon,\\
 \Delta(P_\varepsilon u_\varepsilon)=0&\text{ on }B_\varepsilon(p).
 \end{cases}
 \end{equation*}
Then there exist constants $C>0$ and $\delta>0$ such that
 \begin{equation*}
 \|P_\varepsilon u_\varepsilon\|_{H^1(M)}\leq C\|u_\varepsilon\|_{H^1(M_\varepsilon)}
 \end{equation*}
for all $u\in H^1_0(M)$ and $0<\varepsilon<\delta$.

For any  bounded open interval  $I\subset\mathbb{R}$ whose endpoints do not belong to $\sigma(\Delta_M)$, we let $\Pi$ and $\Pi_\varepsilon$ be the spectral projections of $\Delta_M$ and $\Delta_\varepsilon$ on this interval. Then $\mathrm{rank}\Pi=\dim(\mathrm{range}\Pi)$ is the number of eigenvalues of $\Delta_M$ in $I$.
Theorem \ref{ConN} is a consequence of
\begin{proposition}\label{rank}When $\varepsilon$ is small enough, $\mathrm{rank}\Pi=\mathrm{rank}\Pi_\varepsilon$.
\end{proposition}

\begin{proof}[Proof of Theorem \ref{ConN}]
First take $I=[0, \lambda_k^{\mathcal D}+\frac 1m]$ with $m$ large enough so that the next Dirichlet eigenvalue is greater than $\lambda_k^{\mathcal D}+\frac 1m$. Then Proposition \ref{rank} tells us that as $\varepsilon \to 0$, the number of eigenvalues of $\Delta_\varepsilon$ in the interval $I$ is the same as the Dirichlet Laplacian $\Delta_M$. Similarly one may apply the argument to smaller and smaller  intervals near each $\lambda_k^\mathcal D$ to conclude that  $\Delta_\varepsilon$ and $\Delta_M$ has the same number of eigenvalues in each of these intervals. The conclusion follows.
\end{proof}

Although the settings are different, the proof of Proposition \ref{rank} is very similar to that of Theorem 1.5 in \cite{RT}.
For completeness we include the proof here. The major difference is that we  use the following lemma, which is a consequence of Theorem 3.1 in \cite{RT}:
\begin{lemma}[\cite{RT}]\label{con}Let $F$ be a bounded Borel function on $(0,\infty)$ which is continuous on a neighborhood of $\sigma(\Delta_M)$. Then $F(\Delta_\varepsilon)u_\varepsilon\to F(\Delta_M)u$ in $L^2(M)$ as $\varepsilon\to0$ for all $u\in L^2(M)$.
\end{lemma}
\begin{proof}[Proof of Lemma \ref{rank}]
First take an orthonormal basis $\{u_1,\cdots,u_k\}$ of $\mathrm{range}\Pi$.
Then by Lemma \ref{con}, for $\varepsilon$ small enough, $\|\Pi_\varepsilon (u_i)_\varepsilon-u_i\|_{L^2(M)}<\frac{1}{2}$ for all $1\leq i\leq k$. This implies that $\{\Pi_\varepsilon (u_i)_\varepsilon\}_{i=1}^k$ is  a linear independent set and thus $\rank\Pi_\varepsilon\geq\rank\Pi$ for $\varepsilon$ small enough.

To prove the reverse inequality, we argue by contradiction. Suppose that there exists a sequence $ \varepsilon_n \to 0$ such that $\rank\Pi_{\varepsilon_n}>\rank\Pi$ for all $n$.
Choose $v_n\in\mathrm{range}\Pi_{\varepsilon_n}$ with $\|v_n\|_{L^2(M_{\varepsilon_n})}=1$ and $v_n\perp\mathrm{range}\Pi$. Since $I$ is bounded,   $\{P_{\varepsilon_n}v_n\}_{n=1}^\infty$ is a bounded set in $H^1_0(M)$. By Rellich's theorem, we can assume that $P_{\varepsilon_n}v_n\to w$ in $L^2(M)$ as $n\to\infty$. Then $\|w\|_{L^2(M)}\geq 1$ and
 \begin{equation*}
 \langle w,u_i\rangle_{L^2(M)}=\lim_{n\to\infty}\langle P_{\varepsilon_n}v_n,u_i\rangle_{L^2(M)}=\lim_{n\to\infty}\langle P_{\varepsilon_n}v_n,u_i\rangle_{L^2(M\setminus M_{\varepsilon_n})}=0
 \end{equation*}
for all $1\leq i\leq k$. So $w\perp\mathrm{range}\Pi$. On the other hand, by Lemma \ref{con},
 \begin{equation*}
 \begin{aligned}
 &\lim_{n\to\infty}\|P_{\varepsilon_n}v_n-\Pi w\|_{L^2(M)}=\lim_{n\to\infty}\|P_{\varepsilon_n}v_n-\Pi_{\varepsilon_n}w_{ \varepsilon_n}\|_{L^2(M)}\\
 =&\lim_{n\to\infty}(\|\Pi_{\varepsilon_n}v_n-\Pi_{\varepsilon_n}w_{ \varepsilon_n}\|_{L^2(M_{\varepsilon_n})}+\|P_{\varepsilon_n}v_n\|_{L^2(M\setminus M_{\varepsilon_n})})\\
 \leq&\lim_{n\to\infty}(\|v_n-w\|_{L^2(M_{\varepsilon_n})}+\|P_{\varepsilon_n}v_n -w\|_{L^2(M)}+\|w\|_{L^2(M\setminus M_{\varepsilon_n})})\\
 =&0.
 \end{aligned}
 \end{equation*}
So $w=\Pi w$ which implies $w\in\mathrm{range}\Pi$, a contradiction.
\end{proof}

\textbf{Declarations.} The author has not disclosed any competing interests.


\begin{thebibliography}{99}




\bibitem{Ber}
A. Berdnikov: Bounds on multiplicities of Laplace operator eigenvalues on surfaces. \emph{Journal of Spectral Theory}, 2018, 8(2): 541-554.

\bibitem{Bes}
G. Besson, Sur la multiplicit\'{e} de la premi\`{e}re valeur propre des surfaces riemanniennes, \emph{Annales de l'institut Fourier}, 1980, 30(1): 109-128.

\bibitem{Cheng}
S Y. Cheng: Eigenfunctions and nodal sets, \emph{Commentarii Mathematici Helvetici}, 1976, 51: 43-55.

\bibitem{FRKC}
Chung F R K. Spectral graph theory. \emph{American Mathematical Society}, 1997.

\bibitem{CC}
B. Colbois and Y. Colin de Verdi\`ere: Sur la multiplicit\'{e} de la premi\`{e}re valeur propre d'une surface de Riemann \`{a} courbure constante. \emph{Commentarii Mathematici Helvetici}, 1988, 63: 194-208.

\bibitem{CV0}
Y. Colin de Verdi\`ere: Construction de laplaciens dont une partie finie du spectre est donn\'{e}e. \emph{Annales scientifiques de l'\'{e}cole normale sup\'{e}rieure}, 1987, 20(4): 599-615.

\bibitem{CV1}
Y. Colin de Verdi\`ere: Sur une hypoth\`ese de transversalit\'e d'Arnold. \emph{Commentarii Mathematici Helvetici}, 1988, 63(2): 184-193.

\bibitem{GC}
G. Courtois: Spectrum of manifolds with holes. \emph{Journal of Functional Analysis}, 1995 134(1): 194-221.

\bibitem{JTBD}
J. Dodziuk, T. Pignataro, B. Randol and D. Sullivan: Estimating small eigenvalues of Riemann surfaces. The legacy of Sonya-Kovalevskaya (Cambridge, Mass., and Amherst, Mass., 1985), \emph{Contemporary Mathematics}, 1987, 64: 93-121.

\bibitem{MD}
M. Dahl: Prescribing eigenvalues of the Dirac operator. \emph{Manuscripta Mathematica}, 2005, 118(2): 191-199.

\bibitem{Hass}
A. Hassannezhad: Conformal upper bounds for the eigenvalues of the Laplacian and Steklov problem. \emph{Journal of Functional analysis}, 2011, 261(12): 3419-3436.

\bibitem{HW}
X. He, Z. Wang: Riemannian metrics with prescribed volume and finite parts of Dirichlet spectrum. Preprint.

\bibitem{Her}
J. Hersch: Quatre propri\'{e}t\'{e}s isop\'{e}rim\'{e}triques de membranes sph\'{e}riques homogenes. \emph{CR Acad. Sci. Paris S\'{e}r. AB}, 1970, 270(5).

\bibitem{PJ}
P. Jammes: Prescription de la multiplicit\'{e} des valeurs propres du laplacien de Hodge-de Rham. \emph{Commentarii Mathematici Helvetici}, 2011, 86(4): 967-984.

\bibitem{PJ12}
P. Jammes: Sur la multiplicit\'{e} des valeurs propres du laplacien de Witten. \emph{Transactions of the American Mathematical Society}, 2012, 364(6): 2825-2845.

\bibitem{PJ14}
P. Jammes: Prescription du spectre de Steklov dans une classe conforme. \emph{Analysis} \& \emph{PDE}, 2014, 7(3): 529-550.

\bibitem{Kar}
M. Karpukhin: Upper bounds for the first eigenvalue of the Laplacian on non-orientable surfaces. \emph{International Mathematics Research Notices}, 2016(20): 6200-6209.

\bibitem{KNPP}
M. Karpukhin, N. Nadirashvili, A. Penskoi and I. Polterovich: An isoperimetric inequality for Laplace eigenvalues on the sphere, \emph{Journal of  Differential Geometry}  2021, 118(2):  313-331.

\bibitem{L}
J. Lohkamp: Discontinuity of geometric expansions. \emph{Commentarii Mathematici Helvetici}, 1996, 71(1): 213-228.


\bibitem{Nad}
N. Nadirashvili: Multiple eigenvalues of the Laplace operator. \emph{Mathematics of the USSR-Sbornik}, 1988, 61(1): 225-238.

\bibitem{Nad1}
N. Nadirashvili: Isoperimetric inequality for the second eigenvalue of a sphere. \emph{Journal of Differential Geometry}, 2002, 61(2): 335-340.

\bibitem{NS}
N. Nadirashvili, Y. Sire: Isoperimetric inequality for the third eigenvalue of the Laplace-Beltrami operator on $\mathbb {S}^ 2$. \emph{Journal of Differential Geometry}, 2017, 107(3): 561-571.

\bibitem{RT}
J. Rauch and M. Taylor: Potential and scattering theory on wildly perturbed domains. \emph{Journal of Functional Analysis}, 1975, 18(1): 27-59.

\bibitem{Sev}
B. S\'{e}vennec: Majoration topologique de la multiplicit\'{e} du spectre des surfaces. \emph{S\'{e}minaire de Th\'{e}orie spectrale et g\'{e}om\'{e}trie de Grenoble}, 1993, 12: 29-35.

\bibitem{YY}
P. Yang, S. T. Yau: Eigenvalues of the Laplacian of compact Riemann surfaces and minimal submanifolds. \emph{Annali della Scuola Normale Superiore di Pisa-Classe di Scienze}, 1980, 7(1): 55-63.

\end{thebibliography}
\end{document}